\documentclass[11pt]{article}

\usepackage{enumerate}
\usepackage{amsthm,amsmath,amssymb}
\usepackage{graphicx}
\usepackage{lineno}
\usepackage[colorlinks=true,citecolor=black,linkcolor=black,urlcolor=blue]{hyperref}
\usepackage[english]{babel}
\usepackage{amsfonts}
\usepackage{epsfig, subfigure}
\usepackage{amscd,latexsym}
\usepackage{url}
\usepackage{color}
\usepackage{authblk}
\usepackage{subfigure}

\def\x{\chi_{_{NL}}}
\def\diam{{\rm diam}}

\theoremstyle{plain}
\newtheorem{theorem}{Theorem}
\newtheorem{lemma}{Lemma}
\newtheorem{cor}{Corollary}
\newtheorem{prop}{Proposition}

{\itshape}{\rmfamily}
\newtheorem{remark}{Remark}

\newtheorem{conj}{Conjecture}

\newtheorem{defi}{Definition}

\begin{document}

\title{
 Neighbor-Locating Colorings in Graphs {\color{red}}
}

\author[2]{Liliana Alcon\thanks{Partially supported by PIP 11220150100703CO CONICET, liliana@mate.unlp.edu.ar}}
\author[2]{Marisa Gutierrez\thanks{Partially supported by  PIP 11220150100703CO CONICET, marisa@mate.unlp.edu.ar}}
\author[1]{Carmen Hernando\thanks{Partially supported by projects MTM2015-63791-R (MINECO/FEDER) and Gen. Cat. DGR 2017SGR1336, carmen.hernando@upc.edu}}
\author[1]{Merc\`e Mora\thanks{Partially supported by projects MTM2015-63791-R (MINECO/FEDER), Gen. Cat. DGR 2017SGR1336 and H2020-MSCA-RISE project 734922-CONNECT, merce.mora@upc.edu}}
\author[1]{Ignacio M. Pelayo\thanks{Partially supported by projects MINECO MTM2014-60127-P, ignacio.m.pelayo@upc.edu}}

\affil[1]{Departament de Matem\`atiques, Universitat Polit\`ecnica de Catalunya}

\affil[2]{Centro de Matem\'aticas, Universidad Nacional de La Plata}

\date{}

\maketitle


\begin{abstract}
A \emph{$k$-coloring} of a  graph $G$ is a $k$-partition $\Pi=\{S_1,\ldots,S_k\}$ of $V(G)$ into independent sets, called \emph{colors}.
A $k$-coloring is called \emph{neighbor-locating} if for every pair of vertices $u,v$ belonging to the same color  $S_i$, the set of colors of the neighborhood of $u$ is different from the set of colors of the neighborhood of $v$. 
The \emph{neighbor-locating chromatic number} $\chi _{_{NL}}(G)$ is the minimum  cardinality of a neighbor-locating coloring of $G$.

\vspace{.2cm}
We establish some tight bounds for the neighbor-locating chromatic number of a graph, in terms of its order, maximum degree  and   independence number. 
We determine all connected graphs of order $n\geq 5$ with neighbor-locating chromatic number $n$ or $n-1$.  
We examine the neighbor-locating chromatic number for two graph operations: join and disjoint union,
and also for two graph families: split graphs and Mycielski graphs.

\vspace{1.1cm}\noindent {\it Key words:} coloring; domination; location; vertex partition; neighbor-locating coloring.

\end{abstract}

\section{Introduction}

\vspace{.2cm}
Domination and location in graphs are two important subjects that have received much attention,
usually separately, but sometimes also both together.
There are mainly two types of location, metric location and  neighbor location.
In this work, we are interested in neighbor location, and we explore this concept in the particular context of a special kind of vertex partitions, called colorings.

\vspace{.6cm}
Metric location in sets was simultaneously introduced by P. Slater \cite{slater} and
F. Harary and R. A. Melter \cite{hararymelter} and further studied in different contexts (see  \cite{chmppsw07,hmpsw10}).
In  \cite{heoe04}, M. A. Henning and O. R. Oellermann introduced the so-called \emph{metric-locating-dominating} sets, by merging the concepts of  metric-locating set and  dominating set.

\vspace{.2cm}
In  \cite{ChaSaZh00}, G. Chartrand,  E. Salehi and P. Zhang, brought the notion of metric location to the ambit of vertex partitions, introducing the resolving partitions, also called metric-locating partition, and defining the partition dimension.
Metric location and domination, in the context of vertex partitions, are studied in \cite{hemope16}.
In  \cite{cherheslzh02}, there were introduced  the so-called \emph{locating colorings} considering resolving partitions formed by independents sets.

\vspace{.3cm}
Neighbor location in sets was  introduced by P. Slater in \cite{slater2}.
Given a graph $G$, a set $S\subseteq V(G)$ is a \emph{dominating set} if every vertex not in $S$ is adjacent to some vertex in $S$. 
A set $S\subseteq V(G)$ is a \emph{locating-dominating set} if $S$ is a dominating set and $N(u)\cap S\neq N(v)\cap S$ for every two different vertices $u$ and $v$ not in $S$. 
The \emph{location-domination number} of $G$, denoted by $\lambda(G)$, is the minimum cardinality of a locating-dominating set.
In \cite{chmpp13,hmp14}, bounds for this parameter are given.
In this paper, merging the concepts studied in \cite{cherheslzh02,slater2}, we introduce  the  \emph{neighbor-locating colorings} and  the \emph{neighbor-locating chromatic number}, and  examine this parameter in some families of graphs.

\vspace{.3cm}
The paper is organized as follows.
In Section \ref{lp}, we define the neighbor-locating colorings  and  introduce the neighbor-locating chromatic number of a graph.
In Section \ref{bounds}, bounds for the neighbor-locating chromatic number of a graph are established in terms of its order, maximum degree and independence number.
In Section \ref{sec.ext}, we focus our attention on graphs with neighbor-locating chromatic number close to the order $n$. 
Concretely, we characterize all  graphs with neighbor-locating chromatic number equal to $n$ or to $n-1$.
Section \ref{ss.join} is devoted  to examining the neighbor-locating chromatic number for some graph operations: the join and the disjoint union.
 Section \ref{ss.split} is devoted to studying the  neighbor-locating chromatic number of connected split graphs  and Mycielski graphs.
 Finally, in Section \ref{op}, we pose several open problems.

\vspace{0.1cm}

\vspace{.3cm}
We introduce now some basic terminology.
All the graphs considered are undirected, simple and finite. 
The vertex set and the edge set of a graph $G$ are denoted by $V(G)$ and $E(G)$, respectively.
If $uv\in E(G)$, then we write $u \sim v$.
Let $v$ be a vertex of  $G$.
The \emph{open neighborhood} of $v$ is $\displaystyle N_G(v)=\{w \in V(G):vw \in E(G)\}$,
and the \emph{closed neighborhood} of $v$ is $N_G[v]=N(v)\cup \{v\}$.
The \emph{degree} of $v$ is $\deg_G(v)=|N_G(v)|$.
If $N_G[v]=V(G)$, then $v$ is called \emph{universal}.
If  $\deg_G(v)=1$, then $v$ is called a \emph{leaf}.
An \emph{isolated vertex} is a vertex of degree zero.
Let $W$ be a subset of vertices of a graph $G$.
The open neighborhood of $W$ is $\displaystyle N_G(W)=\cup_{v\in W} N_G(v)$, and the closed neighborhood of $W$ is $N_G[W]=N_G(W)\cup W$.
The subgraph of $G$ induced by $W$, denoted by $G[W]$, has $W$ as vertex set  and $E(G[W]) = \{vw \in E(G) : v \in W,w \in W\}$.
If a graph $H$ is an induced subgraph of $G$, then we write $H \prec G$.

\vspace{.3cm}
The distance between vertices $v,w\in V(G)$ is denoted by $d_G(v,w)$, or $d(v,w)$ if the graph $G$ is clear from the context.
The \emph{diameter} of $G$ is ${\rm diam}(G) = \max\{d(v,w) : v,w \in V(G)\}$.
The \emph{independence number} of $G$, denoted by $\alpha(G)$, is the maximum cardinality of an independent set of $G$.
For undefined terminology,  we refer the reader to \cite{chlezh11}.

\newpage
\section{Locating partitions}\label{lp}

\vspace{.2cm}
In this section, we present different kinds of locating partitions that have been extensively studied in recent years and that are related to the partitions that we introduce in this paper: the neighbor-locating colorings.

Given a connected graph $G$, a vertex $v\in V(G)$ and a set of vertices $S\subseteq V(G)$, the distance $d(v,S)$ between $v$ and  $S$ is $d(v,S)=\min\{d(v,w):w\in S\}$.
Given a  partition $\Pi=\{S_1,\ldots,S_k\}$ of $V(G)$,
we denote by $r(v |\Pi)$  the vector of distances between a vertex $v\in V(G)$ and the elements of  $\Pi$, that is, $r(v | \Pi)=(d(v,S_1),\dots ,d(v,S_k))$.
The partition $\Pi$ is called a  \emph{metric-locating partition},  an \emph{ML-partition} for short,   if, for any pair of distinct vertices $u,v\in V(G)$, $r(u | \Pi)\neq r(v | \Pi)$.
The \emph{partition  dimension} $\beta_p(G)$ of $G$ is the minimum  cardinality of an ML-partition of $G$.
Metric-locating partitions were introduced in \cite{ChaSaZh00}, and further studied in several papers: bounds \cite{chagiha08}, graph families
\cite{fegooe06,ferogo14,gsrm,gsrmw,hb15,jrsa,js08,mv,royele14,si,toim09,tjs07} and  graph operations \cite{bd,cy,da,gyro10,gyjakuta14,ryk,yjkt,yr}.

A partition $\Pi=\{S_1,\ldots,S_k\}$ of $V(G)$ is \emph{dominating} if, for every $i \in \{1,\ldots, k\}$ and for every vertex $v \in S_i$, $d(v,S_j)=1$, for some $j \in \{1,\ldots, k\}$.
The partition $\Pi$ is called a  \emph{metric-locating-dominating partition},  an \emph{MLD-partition} for short,  if it is both dominating and metric-locating.
The \emph{partition metric-location-domination number} $\eta_p(G)$ of $G$ is the minimum  cardinality of an MLD-partition of $G$.
In \cite{hemope16}, it was proved that $\beta_p (G) \le   \eta_p (G) \le  \beta_p (G)+1$.

Let $\Pi=\{S_1,\ldots,S_k\}$ be a partition of $V(G)$. 
If all the parts of  $\Pi$ are independent sets, then we say that $\Pi$ is a \emph{coloring} of $G$ and that the elements of $S_i$ are colored with color $i$. 
The \emph{chromatic number} $\chi(G)$  equals the minimum cardinality  of a  coloring of $G$.

A coloring $\Pi=\{S_1,\ldots,S_k\}$ is called a  \emph{(metric-)locating coloring}, an ML-coloring for short,  if for every $i \in \{1,\ldots, k\}$ and for every pair of distinct vertices $u,v\in S_i$, there exists  $j \in \{1,\ldots, k\}$ such that $d(u,S_j)\neq d(v,S_j)$. In other words, an ML-coloring $\Pi$  is  a coloring that is also an ML-partition.
The \emph{(metric-)locating-chromatic number} $\chi_{_{L}}(G)$  is the minimum  cardinality of an ML-coloring of  $G$. This parameter was introduced in \cite{cherheslzh02} and further studied in
\cite{bean15,beom11,beom16,cherheslzh03,hb15,pbas15,pbasb17,sba15,webasiut13,webasiut15,webasiut17}.

In this paper, we introduce a new type of locating coloring. 
If in the previous paragraph the location was in terms of distances (and just for connected graphs), now we focus our  attention on the neighbors.

\vspace{.2cm}
\begin{defi}
{\rm Let $G$ be a  graph $G$, not necessarily connected. 
	A coloring  $\Pi=\{ S_1,\dots ,S_k\}$ is called a  \emph{ neighbor-locating coloring}, an \emph{NL-coloring} for short, if for every pair of  different vertices $u,v$ belonging to the same color  $S_i$, the set of colors of the neighborhood of $u$ is different from the set of colors of the neighborhood of $v$, that is,
$ \{ j : 1\le j \le k, N(u)\cap S_j\not= \emptyset \} \not= \{j : 1\le j \le k, N(v)\cap S_j\not= \emptyset  \}.$

The \emph{neighbor-locating chromatic number} $\chi _{_{NL}}(G)$, the \emph{NLC-number} for short,  is the minimum  cardinality of an NL-coloring of $G$.}
\end{defi}

\vspace{.1cm}
\begin{remark}
	\rm{Let $\Pi$ be an NL-coloring of a graph $G$.
	If $G$ is a non-connected graph with isolated vertices, then $\Pi$ is not a dominating partition. 
	Conversely, if $G$ is  either a connected graph  or  a non-connected graph without isolated vertices, then  $\Pi$ is also a dominating partition.}
\end{remark}

\vspace{.1cm}
\begin{remark}
\rm{A coloring $\Pi=\{S_1,\ldots,S_k\}$ is an NL-\emph{coloring}  if, for every $i \in \{1,\ldots, k\}$ and for every pair of distinct non-isolated  vertices $u,v\in S_i$, there exists  $j \in \{1,\ldots, k\}$ such that  either $d(u,S_j)=1$ and  $d(v,S_j)\neq1$ or $d(u,S_j)\neq 1$ and  $d(v,S_j)=1$
and there is at most one isolated vertex of color $i$, for every $i\in \{ 1,\dots ,k\}$.}
\end{remark}

\vspace{.1cm}
\begin{remark}\label{z1}
{\rm Let $\Pi= \{ S_1,\dots ,S_k\}$ be a partition of the set $V(G)$ of vertices of a graph $G$.
For every vertex $x\in V(G)$, we define the tuple $nr(x|\Pi)=(x_1,\dots ,x_k)$
as follows}
$$x_i=
\left\{  \begin{array}{ll}0, &\hbox{ if }x\in S_i; \\ 1, &\hbox{ if }x\in N(S_i)\setminus S_i ;
\\ 2, & \hbox{ if }x\notin N[S_i] .
\end{array} \right. $$
{\rm Observe that, if $G$ is connected, then $x_i=\min \{ 2, d(x,S_i) \}$. Notice that  $nr(x|\Pi)$ has exactly one component equal to $0$. In fact, $S_i$ contains exactly all the vertices $x\in V(G)$ such that the i-th component of $nr(x|\Pi )$ is equal to $0$.
With this terminology, $\Pi$ is an NL-coloring if and only if
the sets $S_1,\dots ,S_k$ are independent and $nr(x|\Pi)\not= nr(y|\Pi)$, for every pair of distinct vertices $x$ and $y$.
Moreover, if $\Pi$ is an NL-coloring of $G$ and $x\in V(G)$ is a non isolated vertex, then the tuple $nr(x|\Pi)$ has at least one component equal to $1$. }
\end{remark}

Given a graph $G$, an NL-coloring $\Pi$ of $V(G)$ and two vertices $x,y \in V(G)$, if $nr(x|\Pi)\not= nr(y|\Pi)$, then  $x,y$ are said to be \emph{neighbor-located} by $\Pi$.

\vspace{.2cm}
\begin{remark}\label{z1}
\rm{If $G$ is a graph and $W\subseteq V(G)$ is the set of all isolated vertices of $G$, then}
$$  \chi _{_{NL}}(G)= \hbox{max}\{\chi _{_{NL}}(G[V(G)\setminus W]), \, |W|\}. $$
\end{remark}

The following chains of inequalities hold.

\begin{prop}\label{z0}
Let $G$ be a non-trivial graph. Then,
\begin{enumerate}[(1)]
\item
$2 \le \chi(G) \le \chi _{_{L}}(G) \le \chi _{_{NL}}(G) $.
\item
$2 \le \beta_p(G) \le \eta_p (G) \le \chi _{_{NL}}(G)$.
\item
$\chi _{_{NL}}(G) \le \chi(G)+ \lambda(G)$.

\end{enumerate}
\end{prop}
\begin{proof}
Items (1) and (2)  are a direct consequence of the definitions.
To prove  (3), take  a minimum locating-dominating set $W$ of $G$. 
Let H be the subgraph of $G$ induced by $V(G)\setminus W$. 
Take  a $k$-coloring $\Pi_{H}=\{S_1, \dots, S_{k}\}$ of   $H$, where $k=\chi(H)$. 
Then, $\Pi=\Pi_{H}\cup \{\{w\} \ : \ w \in W\}$ is an NL-coloring of $G$. As $\chi(H) \leq \chi(G)$, we obtain that the inequality  is satisfied.
\end{proof}

\begin{figure}[!ht]
\begin{center}
\includegraphics[width=\textwidth]{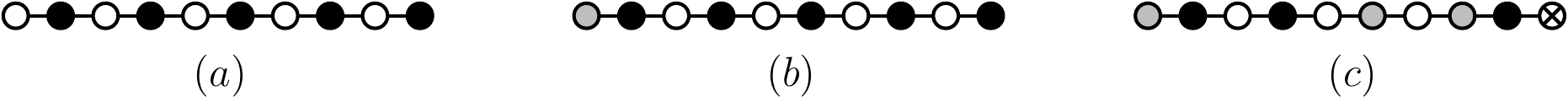}
\caption{
(a) $\chi(P_{10})=2$; (b) $\chi _{_{L}}(P_{10})=3$ and (c) $\chi _{_{NL}}(P_{10})=4$.}
\label{P10}
\end{center}
\end{figure}

{\begin{prop}
	For each pair $h$, $k$ of integers with $3 \le  h \le k$, there exists a connected graph $G$   with $\chi(G)=h$ and $\chi _{_{NL}}(G)=k$.
\end{prop}
	\begin{proof}
		It is enough to consider the graph obtained from the complete graph $K_h$ by hanging $k$-1 leaves to one of its vertices.
	\end{proof}

\newpage
\section{Bounds}\label{bounds}

This section is devoted to establishing some bounds involving the NLC-number, the order, the diameter and the independence number of a graph.
We begin with some properties of the NLC-number depending on the diameter of the graph.

\begin{prop}\label{d2d4}
	Let $G$ be a  connected graph of order $n\ge3$.
	\begin{enumerate}[{\rm (1)}]
		\item If ${\rm diam}(G)\le2$, then  $\chi _{_{L}}(G)=\chi _{_{NL}}(G)$.
		\item If ${\rm diam}(G)\ge4$, then  $\chi _{_{NL}}(G)\le n-2$.
	\end{enumerate}
\end{prop}
\begin{proof}
	\begin{enumerate}[{\rm (1)}]
\item
Let $\Pi=\{ S_1,\dots ,S_k \}$ be an ML-coloring of $G$.
	If $\diam (G)=2$, then for every $x\in V$ we have $\min \{ 2, d(x,S_i) \}=d(x,S_i)$.
	Hence, $nr(x|\Pi)=r(x|\Pi)$, for every $x\in V(G)$,  that is, $\Pi$ is also an NL-coloring of $G$.
	Thus, $\x (G)\le \chi_{_{L}} (G)$, and by Proposition~\ref{z0}, we have $\x (G) = \chi_{_{L}} (G)$.


\item If $\diam (G)\ge 4$, then there exist vertices $u,v\in V(G)$ such that $d(u,v)=4$.
		Take $x_1,x_2,x_3\in V(G)$ such that the set $\{u,x_1,x_2,x_3,v\}$ induces a shortest path joining vertices $u$ and $v$, and $d(u,x_3)=d(u,x_2)+1=d(u,x_1)+2=3$.
		The partition $\Pi=\{S_1,S_2, S_3\} \cup \{ \{ z \}: z\notin S_1\cup S_2\cup S_3 \}$, where $S_1=\{u,x_3\}$, $S_2=\{v,x_1\}$ and $S_3=\{x_2\}$, is clearly a coloring of $G$. We claim that $\Pi$ is also an NL-coloring. Indeed,
		$nr(u|\Pi)=(0,1,2,\dots )\not= (0,1,1,\dots )=nr(x_3|\Pi)$ and $nr(v|\Pi)=(1,0,2,\dots )\not= (1,0,1,\dots )=nr(x_1|\Pi)$.
Hence,  $\chi _{_{NL}}(G)\le n-2$.
	\end{enumerate}
	\vspace{-0.5cm}\end{proof}

\begin{theorem}\label{gb}
Let $G$ be a non-trivial graph of order $n$ and maximum degree $\Delta$.
If $\chi _{_{NL}}(G)=k$ and $\Delta\le k-1$, then

\begin{enumerate}[{\rm(1)}]
\item [{\rm(1)}]$\displaystyle{n\le k\, 2^{k-1}}$ \hspace{.1cm} and \hspace{.1cm} {\rm(2)}
  $\displaystyle{n\le k\, \sum_{j=0}^{\Delta } \binom{k-1}{j}}.$
\end{enumerate}
In addition, if $G$ has no isolated vertices, then:
\begin{enumerate}[{\rm(1)}]
\item[{\rm(3)}] $\displaystyle{n\le k\, (2^{k-1}-1)}$ \hspace{.1cm} and \hspace{.1cm} {\rm(4)}
 $\displaystyle{n\le k\, \sum_{j=1}^{\Delta } \binom{k-1}{j}}.$
\end{enumerate}
\end{theorem}
\begin{proof} Let $\Pi= \{ S_1,\dots ,S_k\}$ be a k-NL-coloring of $G$.  An upper bound of the order of $G$ is given by the maximum number of
suitable $k$-tuples $nr(x|\Pi)$, $x\in V(G)$.

\begin{enumerate}
\item[{\rm(1)}]
There are at most $2^{k-1}$ possible tuples  with  the $i$-th component equal to $0$ and the remaining components equal to $1$ or $2$.
Therefore, $|S_i|\le 2^{k-1}$.
Hence, $n=|V(G)|=\sum_{i=1}^k |S_i|\le \sum_{i=1}^k 2^{k-1}=k\, 2^{k-1}$.

\item[{\rm(2)}]
If  $\Delta \le k-1$ and $x\in S_i$, then the $i$-th component of the $k$-tuple $nr(x|\Pi)$ is $0$ and at most $\Delta $ components are equal to 1. Therefore,
$|S_i|\le \sum_{j=0}^{\Delta } \binom{k-1}{j}$, and the upper bound follows.

\item[{\rm(3)}]
In this case,  for every $x\in V(G)$, the k-tuple $nr(x|\Pi)$ has at least one component          equal to $1$.
There are  $2^{k-1}-1$ $k$-tuples  with  the $i$-th component          equal to $0$ and the remaining components equal to $1$ or $2$, but not all them equal to $2$.
Hence, $n=|V(G)|=\sum_{i=1}^k |S_i|\le \sum_{i=1}^k (2^{k-1}-1)=k\, (2^{k-1}-1)$.

\item[{\rm(4)}]
If $\Delta \le k-1$ and $G$ has no isolated vertices, then, for every $x\in S_i$, the $i$-th component of the $k$-tuple $nr(x|\Pi)$ is $0$, and the number of components which are equal to $1$ is at least $1$ and at most  $\Delta $. 
Therefore, $|S_i|\le \sum_{j=1}^{\Delta } \binom{k-1}{j}$, and the upper bound follows.
\end{enumerate}
\vspace{-.8cm}\end{proof}

Notice that the bounds displayed in items (3) and (4) of Theorem \ref{gb} apply also for connected graphs.

Next, for every integer $k\ge 3$, we build a connected graph $G_k=(V_k,E_k)$
of maximum order with NLC-number $k$.
The set $V_k$ of vertices of $G_k$ is the set of all words of length $k$ in the alphabet $\{ 0,1,2 \}$ having exactly one $0$ and at least one $1$. 
To define the edges of $G_k$, let $W_i$ be the set of words $x_1\dots x_k\in V_k$ such that $x_i=0$, for every $i\in \{ 1,\dots ,k\}$, so that $\{W_1,\dots ,W_k\}$ is a partition of $V_k$.
For every $x,y\in V_k$, if $x=x_1\dots x_k\in W_i$ and $y=y_1\dots y_k\in W_j$, then $xy\in E_k$ if and only if $i\neq j$, $x_j=1$ and $y_i=1$ (see an illustration of graph $G_3$ in Figure~\ref{G3}).

\begin{figure}[!ht]
\begin{center}
\includegraphics[width=0.60\textwidth]{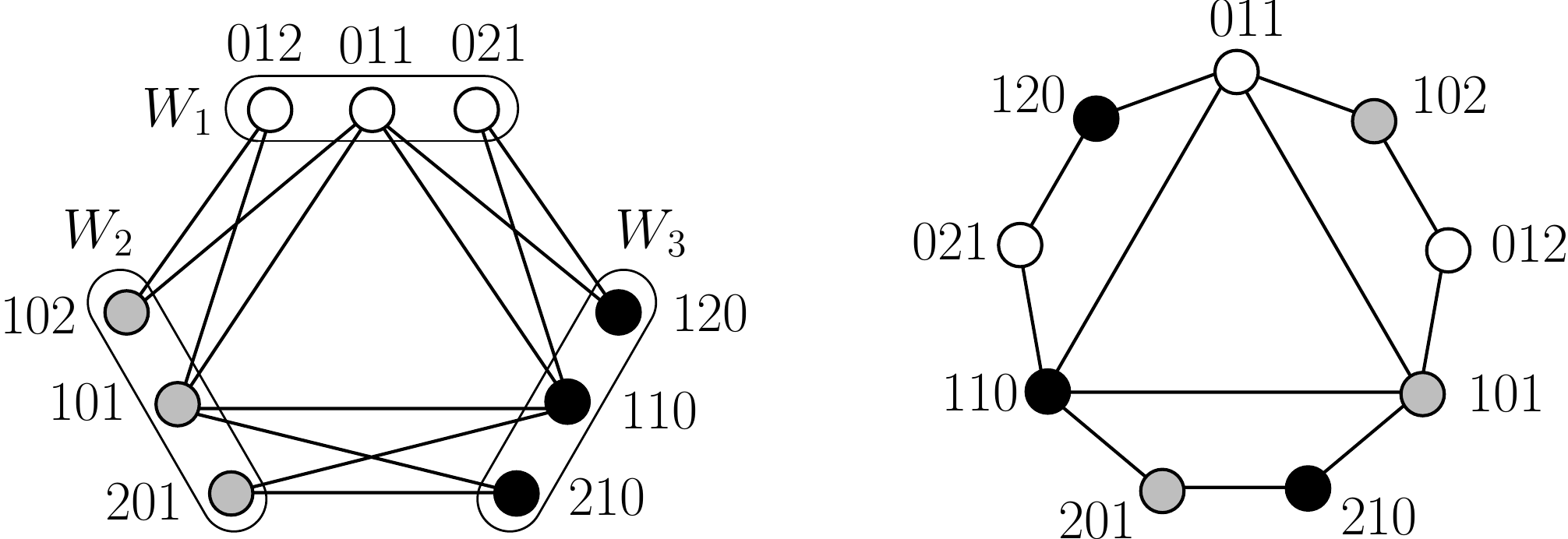}
\caption{Two representations of the graph $G_3$.}\label{G3}
\end{center}
\end{figure}

\vspace{.3cm}\noindent Let $n_{{_k}}=k\cdot (2^{k-1}-1)$, $m_{{_k}}=k\cdot (k-1) \cdot 2^{2k-5}$,  $\delta_{{_k}}=2^{k-2}$ and $\Delta_{{_k}}=(k-1)\cdot 2^{k-2}$.

\begin{prop}\label{Gk}
For every integer  $k\ge 3$,
$G_k$ is a connected graph of
order $n_{{_k}}$,
size $m_{{_k}}$,
diameter $3$, minimum degree $\delta_{{_k}}$, maximum degree $\Delta_{{_k}}$
such that $\chi _{_{NL}}(G_k)=k$.
\end{prop}

\begin{proof}
For every $i\in \{ 1,\dots, k\}$, we have $|W_i|=2^{k-1}-1$ because
there are $2^{k-1}$ words in the alphabet $\{ 0,1,2 \}$ with exactly one 0 in the position $i$, and only one of these words has no 1's.
Hence, $|V_k|=k\cdot (2^{k-1}-1)=n_k$.

Let $x\in W_1$ and $y\in W_2$.
Notice that $xy\in E_k$ if and only if $x=01x_3\ldots x_k$ and $y=10y_3 \ldots y_k$.
Hence, the number of edges with an endpoint in $W_1$ and the other in $W_2$ is $(2^{k-2})^2=2^{2k-4}$.
By symmetry, the number of edges with an endpoint in $W_i$ and the other in $W_j$ is $2^{2k-4}$ for every pair $i,j\in \{1,\dots ,k\}$ with $i\not=j$.
Since the sets $W_i$ are independent in $G_k$, we have $|E_k|=\binom{k}{2}2^{2k-4}=k\cdot (k-1) \cdot 2^{2k-5}=m_k$.

Now, let $x\in W_i$.  If $x_j=1$, then $x$ is adjacent to every vertex $y\in W_j$ such that $y_i=1$. There are $2^{k-2}$ such vertices in $W_j$.
Thus, $\deg (x)= |\{ j: x_j=1\}| \, 2^{k-2}$.
Hence, the minimum degree is attained by the vertices $x=x_1\dots x_k\in V_k$ such that $ |\{ j: x_j=1\}|=1$, whereas the maximum degree is attained when $ |\{ j: x_j=1\}|=k-1$.
Therefore, we have $\delta_k= 2^{k-2}$ and $\Delta_k=(k-1)2^{k-2}$.

Let $k\ge 3$. To prove that $G_k$ has diameter $3$, we show that the distance between any two different vertices
$x,y\in V_k$ is at most 3 and at least two of them are at distance 3. We distinguish two cases.

If $x,y\in W_i$ for some $i\in \{1,\dots ,k\}$, then we may assume without loss of generality that  $x=0x_2\dots x_k\in W_1$ and $y=0y_2\dots y_k\in W_1$.
Observe that $d(x,y)=2$ if and only if $x$ and $y$ have a common neighbor and, by definition of $G_k$, this happens if and only if $x_j= y_j=1$ for some $j\in \{2,\dots ,k\}$. 
If this condition does not hold, then $x_h=y_l=1$ and $x_l=y_h=2$, for some $h,l\in \{2,\dots ,k\}$.  
Without loss of generality, we may assume $h=2$ and $l=3$. 
In such a case,  $x=012x_4\dots x_k$, $y=021y_4\dots y_k$, and $d(x,y)=3$ because:
$$x=012x_4\dots x_k\sim 101y_4\dots  y_k\sim 110y_4\dots y_k \sim 021y_4\dots y_k=y.$$

If $x\in W_i$ and $y\in W_j$ for some $i,j\in \{1,\dots ,k\}$, $i\not= j$, then we may assume by symmetry that  $x=0x_2x_3\dots x_k\in W_1$ and $y=y_10y_3\dots y_k\in W_2$.
If $x_2=y_1=1$, then $xy\in E_k$ and $d(x,y)=1$.
If $x_2=1$ and $y_1=2$, then $y_l=1$ for some $l\in \{3,\dots ,k\}$, and $d(x,y)\le 3$ since:
$$x=01x_3\ldots x_k \sim 10y_3\ldots \stackrel{l)}{1} \ldots y_k \sim 21y_3\ldots \stackrel{l)}{0} \ldots y_k \sim 20y_3\ldots \stackrel{l)}{1} \ldots y_k =y.$$
In a similar way, we can prove that $d(x,y)\le 3$ if $x_2=2$ and $y_1=1$.
It only remains to consider the case $x_2=y_1=2$.
If $x_l=y_l=1$ for some $l\in \{ 3,\dots ,k\}$, then $d(x,y)\le 2$ because:
$$x=02x_3\ldots \stackrel{l)}{1} \ldots x_k\sim 11y_3\ldots \stackrel{l)}{0} \ldots y_k\sim 20y_3\ldots \stackrel{l)}{1}\ldots y_k=y.$$
Otherwise, $k\ge 4$ and there exist $h,l\in \{ 3,\dots ,k\}$ such that $x_h=y_l=1$ and $x_l=y_h=2$. We may assume without loss of generality that
$h=3$ and $l=4$. Then, $d(x,y)\le 3$ since
$$x=0212x_5\dots x_k\sim 1101y_5\dots y_k \sim 1110y_5\dots y_k \sim 2021y_5\dots y_k.$$

Finally, to see that $\chi _{_{NL}}(G_k)=k$, observe first that, according to Theorem \ref{gb} (3),
since $G_k$ is a connected graph of order $|V_k|=n_k=k\cdot (2^{k-1}-1)>(k-1) (2^{(k-1)-1}-1)$, we have $ \chi _{_{NL}}(G_k)\ge k$.
Besides,  $\{W_1,\ldots, W_k\}$ is clearly an NL-coloring of $G_k$.
Indeed, $W_1,\ldots, W_k$ are independent sets. 
Moreover, if $x$ and $y$ are two different vertices of $W_i$, then $x_j\not= y_j$ for some $j\not= i$. Hence, $\{x_j,y_j\}=\{ 1,2\}$. 
We may assume without loss of generality that $x_j=1$ and $y_j=2$.  Then, $x$ has a neighbor in $W_j$, but $y$ has no neighbor in $W_j$.
Therefore, $\{W_1,\ldots, W_k\}$ is an NL-coloring, implying that $\chi _{_{NL}}(G_k) = k$.
\end{proof}

\begin{cor}
The bounds displayed in items {\rm(1)} and {\rm(3)} of Theorem \ref{gb} are tight, for every $k\ge 3$.
\end{cor}
\begin{proof}
For every $k\ge 3$, the graph $G_k$ attains the bound given in Theorem \ref{gb} (3). If we add $k$ isolated vertices to $G_k$, then we have a graph attaining the bound given in Theorem \ref{gb} (1).
\end{proof}

\begin{prop}\label{subgrafo}
Let $k\ge 3$ and $\mu_{{_k}}=k\, (k-1) \, 2^{k-3}$.
If $H$ is a graph with no isolated vertices such that $\chi _{_{NL}}(H)=k$,
then

\begin{enumerate}[\rm (1)]

\item $H$ is isomorphic to a subgraph of $G_k$.

\item If $H$ has order $n_k$, then $\mu_k \le |E(H)| \le m_k$.

\end{enumerate}
\end{prop}
\begin{proof}
Let $\Pi=\{ S_1, \dots , S_k\}$ be an NL-partition of $H$.
Recall that if $x\in V(H)$, then the tuple $nr(x|\Pi)=(x_1,\dots ,x_k)$ satisfies $x_i=0$ if  $x\in S_i$, and $x_j\in \{1,2\}$, if $j\not=i$.
Moreover,  $x_j=1$ for some $j\not= i$, since $H$ has no isolated vertices.

To prove item (1),
we identify $x\in V(H)$ with vertex $x_1\dots x_k\in V(G_k)$, whenever $nr(x|\Pi)=(x_1,\dots ,x_k)$.
If $xy\in E(H)$, then $x\in S_i$ and $y\in S_j$, for some $i,j\in \{1, \dots ,k \}$ with $i\not=j$.
Thus, if $nr(x|\Pi)=(x_1,\dots ,x_k)$ and $nr(y|\Pi)=(y_1,\dots ,y_k)$, then we have $x_i=y_j=0$ and $x_j=y_i=1$, that is, $x_1\dots x_k$ and $y_1\dots y_k$ are adjacent in $G_k$. 
Hence, $H$ is isomorphic to a subgraph of $G_k$.

Since $H$ is isomorphic to a subgraph of $G_k$, we have $|E(H)|\le E(G_k)\le m_k$. 
Hence, the upper bound of item (2) holds. 
To prove the lower bound,
notice that if $|V(H)|=|V(G)|=n_k$, then $S_i$ has $2^{k-1}-1$ vertices for every $i\in \{ 1,\dots ,k\}$ and $2^{k-2}$ of them are adjacent to a vertex of $S_j$, if $j\not= i$. 
Hence, fixed $i,j\in \{1,\dots ,k\}$ with $i\not= j$, the number of edges with an endpoint in $S_i$ and the other in $S_j$ is at least $2^{k-2}$. 
Therefore, $|E(H)|\ge \binom{k}{2} 2^{k-2}=k(k-1) 2^{k-3}=\mu_k$.
\end{proof}

\begin{remark}{\rm
	As a consequence of Proposition~\ref{subgrafo}, if $H$ is a graph of order at most $n_{{_k}}$ without isolated vertices not isomorphic to any subgraph of $G_k$, then  $\chi _{_{NL}}(H)\geq k+1$.
	However, the converse is not true: there are subgraphs of $G_k$ with NLC-number greater than $k$.
	For example, the cycle of order 4, $C_4$,  is a subgraph of $G_3$, but $\chi _{_{NL}}(C_4)=4$.}
\end{remark}

Next, we relate the NLC-number $\chi _{_{NL}}(G)$ to the independence number $\alpha(G)$ of a twin-free graph $G$ .

\begin{prop}
	If $G$ is a twin-free graph of order $n$, then $\chi _{_{NL}} (G) \le n-\alpha(G)+1$. 
	Moreover, this bound is tight.
\end{prop}
\begin{proof}
	Let $\Omega$ be a maximum independent set of $G$, that is, an independent set such that $|\Omega|=\alpha(G)$.
	Consider the partition $\Pi= \{ \Omega \} \cup\{ \{z\} : z\in V(G)\setminus \Omega \}$.
	Notice that, since $G$ is a twin-free graph,  the partition $\Pi$ is an NL-coloring of $G$.
	As $| \Pi|=n-\alpha(G) +1$,  $\chi _{_{NL}} (G) \le n-\alpha(G)+1$.

To prove the tightness of the bound, let $H$ be the graph obtained from the complete graph $K_r$ by attaching $r-1$ leaves to respectively $r-1$ different vertices of $K_r$. Then, $H$ is a connected twin-free graph of order  $2r-1$, such that $\alpha (H)=r$. Besides, $ \x (H)\ge \chi (H) = r$ and
it is easy to check that $\x (H)=r$. Indeed, any partition with all parts but one of size two, such that each part of size two contains a leaf $u$ together with a vertex of $K_r$ non-adjacent to $u$, is an r-NL-coloring of $H$.
Hence,
$\chi _{_{NL}} (H) = r = (2r-1)-r+1=|V(H)|-\alpha (H)+1$.
\end{proof}

\newpage
\section{Extremal graphs}\label{sec.ext}

In this section, we focus our attention on graphs with NLC-number close to the order.
In \cite{cherheslzh02,cherheslzh03}, all connected graphs of order $n$ and $\chi _{_{L}}(G)=n$ and $\chi _{_{L}}(G)=n-1$ were characterized. 
Now, we approach the same problems for $\chi _{_{NL}}(G)$. 
In fact, we show that the graphs achieving these extreme values are the same for both parameters.

\begin{theorem}  [\cite{cherheslzh02}]\label{xmln}
	 If $G$ is a  connected graph of order $n\ge3$, then
	 $\chi _{_{L}}(G)=n$ if and only if $G$ is a complete multipartite graph.
\end{theorem}

Notice that complete graphs $K_n$, complete bipartite graphs $K_{h,n-h}\cong \overline{K_h}\vee \overline{K_{n-h}}$, stars $S_{1,n-1}\cong \overline{K_1}\vee \overline{K_{n-1}}$ and complete split graphs $\overline{K_h}\vee K_{n-h}$ are  some examples of complete multipartite graphs.

\begin{theorem} \label{xnln}
	If $G$ is a  graph of order $n\ge3$,
	then $\chi _{_{NL}}(G)=n$ if and only if $G$ is either a complete multipartite graph or $G=\overline{K_n}$.
\end{theorem}
\begin{proof}
	Clearly, $\chi _{_{NL}}(\overline{K_n})=n$. 
	By other hand, if $G$ is a complete multipartite graph, then according to Theorem \ref{xmln}, $\chi _{_{L}}(G)=n$.
	Thus, by Proposition \ref{z0}, $\chi _{_{NL}}(G)=n$.
	
	Conversely, we distinguish two cases depending of  whether or not the graph $G$ is connected.

If $G$ is a connected graph of order $n\ge3$  which is not a complete multipartite graph, then there exists a pair of non-adjacent vertices $u,v\in V(G)$ such that $N(u)\neq N(v)$.
	Let $\Pi=\{S_1,S_2,\ldots,S_{n-1}\}$ be the coloring of $G$ such that $S_1=\{u,v\}$ and, for every $i\neq 1$, $|S_i|=1$.
	Certainly, $\Pi$ is an NL-coloring, since $N(u)\neq N(v)$.
	 Without loss of generality, we can consider $z \in V(G)$ such that $uz \in E(G)$ and $vz \notin E(G)$, so $d(u, \{z\})\neq d(v, \{z\})$ and $nr(u|\Pi)\neq nr(v|\Pi)$, and thus $\chi _{_{NL}}(G)\le n-1$.

If $G$ is a non-connected graph of order $n\ge3$  other than $\overline{K_n}$, then there exists a pair of adjacent vertices $x,y\in V(G)$ and there exist a vertex $t$ in a  connected component different from that of $x$ and $y$.
Let $\Pi=\{S_1,S_2,\ldots,S_{n-1}\}$ be the coloring of $G$ such that $S_1=\{x,t\}$ and, for every $i\neq 1$, $|S_i|=1$.
Clearly,  $\Pi$ is an NL-coloring, since $nr(x | \Pi)\neq nr(t | \Pi)$, and thus $\chi _{_{NL}}(G)\le n-1$.
\end{proof}

We next study the graphs of order $n$ and  NLC-number $n-1$. 
For this, we introduce first some families of graphs that will play an important role.

\begin{itemize}

	\item Let $\cal H$ denote the set of all connected graphs $G$ of order $n\ge3$ such that, for some vertex $v \in G$, $G-v$ is a complete multipartite graph.

	\item For $G\in  \cal H$,  call  $V_1$, $V_2$, $\ldots$, $V_k$, $k \geq 2$, to  the partite sets of $G-v$; and let $n_i=|V_i|$ and $a_i=|N(v)\cap V_i|$  for  $1\le i \le k$.
	
	\item Let $\cal F$ denote the set of all graphs $G\in {\cal H}$ satisfying  al least one of the following two  properties:

	\begin{enumerate}
		
		\item[(1)]  $a_i\in \{0,n_i\}$ for every   $i \in \{1,\ldots,k\}$, and  $| \{ i \in \{1, \ldots ,k\} \, | \, a_i=0\}|\geq 2$.

		\item[(2)] There is exactly one integer $i \in \{1,\ldots,k\}$ such that $a_i \not\in \{0,n_i\}$, and $a_i=n_i-1$ for this integer $i$.

	\end{enumerate}
	
	\item Let $\cal G$ denote the set of all graphs of order $n$ that are the join of $2K_2$ and a complete multipartite graph $G^*$ of order $n-4 \ge 1$, that is, $G= G^*\vee 2K_2$, $V(G)=V_1 \cup V_2$, $V_1=\{v_1, \dots, v_4\}$, $V_2=\{v_5, \dots, v_n\}$, $G[V_1]=2K_2$ and $G[V_2]=G^*$
and all the edges that connect vertices of $V_1$ with  vertices of $V_2$ are in $E(G)$ (see next section for more properties  of join graphs).

\end{itemize}

\begin{theorem}  [\cite{cherheslzh03}]\label{xmln-1}
	If $G$ is a connected graph of order $n\ge4$, then $\chi _{_{L}}(G)=n-1$ if and only if $G \in {\cal F} \cup {\cal G}$.
\end{theorem}

\begin{lemma} \label{2k2}
	If $G$ is a  graph of order $n\ge5$,  $\chi _{_{NL}}(G)=n-1$ and $2K_2 \prec G$, then  $G \in {\cal G}$.
\end{lemma}
\begin{proof}
	Let $S=\{a_1,a_2,b_1,b_2\}\subset V(G)$ be a set such that $a_1a_2,b_1b_2\in E(G)$ and $G[S]\cong 2K_2$.
	Let $w\in V(G) \setminus S$ and let $h=|N(w)\cap S|$.
	Suppose that $h\le 3$ and $wb_2\not\in E(G)$.
	Consider the $(n-2)$-partition $\Pi=\{S_1,S_2,S_3,\ldots,S_{n-2}\}$, where $S_1=\{w,b_2\}$, $S_2=\{a_1,b_1\}$ and $S_3=\{a_2\}$.
	Notice that $\Pi$ is an NL-coloring of $G$, except in two cases: that $wa_2 \notin E(G)$ and also $w$ is adjacent to $a_1$ or to $b_1$. In the case that $wa_2 \notin E(G)$ and $wa_1 \in E(G)$ we take $S_1=\{w,b_2\}$, $S_2=\{a_2,b_1\}$ and $S_3=\{a_1\}$. If $wa_2 \notin E(G)$ and $wb_1 \in E(G)$ we take $S_1=\{w,a_2\}$, $S_2=\{a_1,b_2\}$ and $S_3=\{b_1\}$.
	Thus, $\chi _{_{NL}}(G)\le n-2$, a contradiction.
	So, we have proved that each vertex of $V(G) \setminus S$ is adjacent to every vertex of $S$.
	
	Take $H=G[V(G) \setminus S]$.
	Suppose that $H$ is not a complete multipartite graph.
	Let $u,v\in V(G) \setminus S$ such that $uv\not\in E(G)$ and $N(u)\not=N(v)$.
	Consider the $(n-2)$-partition $\Pi=\{S_1,S_2,S_3,\ldots,S_{n-2}\}$, where $S_1=\{u,v\}$ and  $S_2=\{a_1,b_1\}$.
	Certainly, $\Pi=$ is  an NL-coloring of $G$, i.e.,  $\chi _{_{NL}}(G)\le n-2$, a contradiction.
	Hence, $H$ is a complete multipartite graph.
\end{proof}

\begin{lemma} [\cite{cherheslzh03}]\label{lemxmln-1}
	If  $G$ is a  connected graph of order $n\ge4$ with
	 $\chi _{_{L}}(G)=n-1$ and $2K_2 \not\prec G$, then  $G \in {\cal F}$.
\end{lemma}

\begin{lemma}   \label{P3K1}
	Let $G$ be a  graph of order $n\ge5$.
	Let $S=\{u_1,u_2,u_3,v\}$ be a set of vertices of $V(G)$ such that $d(u_1,u_3)=2$, $u_2\in N(u_1)\cap N(u_3)$ and $N(u_1) \neq N(u_3)$.
	If $S$ induces a subgraph of $G$ isomorphic to  $P_3 + K_1$, then $\chi _{_{NL}}(G)\le n-2$.
\end{lemma}
\begin{proof}
	Let $w$ be a vertex of $N(u_1)\cup N(u_3)$ not belonging to $N(u_1)\cap N(u_3)$.
	Take the partition $\Pi=\{S_1, S_2,S_3,\ldots,S_{n-2}\}$ such that $S_1=\{u_1,u_3\}$, $S_2=\{u_2,v\}$ and $S_3=\{w\}$.
	Notice that $d(u_1,S_3)=1<d(u_3,S_3)$ and $d(u_2,S_1)=1<d(v,S_1)$.
	Hence, $\Pi$ is an NL-coloring of $G$.
\end{proof}

\begin{lemma}\label{diam3}
	Let $G$ be a   graph of order $n\ge5$, diameter $3$ such that $2K_2 \not\prec G$.
	If  $\chi _{_{NL}}(G)=n-1$, then $\chi _{_{L}}(G)=n-1$.
\end{lemma}
\begin{proof}
	Suppose on the contrary that there exist a graph $G$ such that $\chi _{_{NL}}(G)=n-1$ and $\chi _{_{L}}(G)\le n-2$.
	Let $\Pi=\{S_1,S_2,\ldots S_{n-2}\}$ be an ML-coloring of cardinality $n-2$. $\Pi$ can not be an NL-coloring.
	If $S_1=\{u_1,v_1\}$,  $S_2=\{u_2,v_2\}$,  $S_3=\{w_3\}$, $\ldots$, $S_{n-2}=\{w_{n-2}\}$, then  we can assume without loss of generality that, for every $j\in\{2,\ldots,n-2\}$,  $nr(u_1 | \Pi)= nr(v_1 | \Pi)$.
	This means that, for every $j\in\{2,\ldots,n-2\}$, either $d(u_1,S_j)=d(v_1,S_j)=1$ or
	$2 \le  d(u_1,S_j),d(v_1,S_j) \le 3$.
	Notice that if, for every $j\in\{3,\ldots,n-2\}$, we have $d(u_1,S_j)=d(v_1,S_j)=1$, then $1 \le  d(u_1,S_j),d(v_1,S_j) \le 2$.
	Hence, we can suppose without loss of generality  that $d(u_1,w_3)=2$ and $d(v_1,w_3)=3$.
	Let $z\in V(G)$ be such that $d(u_1,z)=d(z,w_3)=1$.
	If, for some $j\in\{4,\ldots,n-2\}$,  we have $z=w_j$, then $d(v_1,z)=1$, and thus $d(v_1,w_3)=2$, a contradiction.
	So, we can suppose without loss of generality  that $z=u_2$.
	Let $x\in V(G)$ be such that $d(v_1,x)=1$ and $d(x,w_3)=2$.
	Notice that $x=v_2$, as otherwise, if for some $j\in\{4,\ldots,n-2\}$, $x=w_j$, then $d(u_1,x)=1$, and according to Lemma \ref{P3K1}, $d(x,w_3)=1$ since $N(u_1)=N(w_3)$, a contradiction.
	Hence, the subgraph induced by $\{u_2,w_3,v_1,v_2\}$ is isomorphic to $2K_2$, which is again a contradiction.
\end{proof}


\vspace{.6cm}
\begin{theorem}  \label{xnln-1}
	Let $G$ be a   graph of order $n\ge5$.
	Then, $\chi _{_{NL}}(G)=n-1$ if and only if either $G \in {\cal F} \cup {\cal G}$ or $G \cong H + K_1$, where $H$ is an arbitrary complete multipartite graph.
\end{theorem}
\begin{proof}
	If  $G \in {\cal F} \cup {\cal G}$,  then, according to Theorem \ref{xmln-1}, $\chi _{_{L}}(G)=n-1$.
	This means that  $\chi _{_{NL}}(G)\ge n-1$, since $\chi _{_{L}}(G)\le \chi _{_{NL}}(G)$.
	Hence, from Theorem \ref{xmln}, we derive that $\chi _{_{NL}}(G)= n-1$.
	
	Let $H$ be a complete bipartite graph of order $n-1$. 
	According to Theorem \ref{xnln}, $\chi _{_{NL}}(H)= n-1$. 
	Let $G \cong H + K_1$ such that $V(K_1)=\{u\}$. 
	If $v\in V(H)$, then it is straightforward to check that the $(n-1)$-coloring of $G$ $\Pi=\{S_1,\ldots,S_{n-1}\}$ such that $S_1=\{u,v\}$, is an NL-coloring of $G$. 
	Thus, $\chi _{_{NL}}(G)= n-1$.

	Conversely, let $G$ be a  graph such that $\chi _{_{NL}}(G)=n-1$.
	We distinguish two case depending  on  whether or not the graph $G$ is connected.
	
	Suppose that  $G$ is a connected graph.
	By Theorem \ref{xnln} and Proposition \ref{d2d4}, it follows that $2 \le {\rm diam} (G) \le 3$.
	If either $\rm diam(G)=2$ or $2K_2 \prec G$, then according to Proposition \ref{d2d4}, Theorem \ref{xmln-1} and Lemma \ref{2k2}, we derive that $G \in {\cal F} \cup {\cal G}$.
	If  $G$ is a graph of diameter  $\rm diam(G)=3$ such that $2K_2 \not\prec G$, then from Lemma \ref{diam3} and Lemma \ref{lemxmln-1}, it follows that $G \in {\cal F}$.
	
	Assume that $G$ is a non-connected graph. 
	We distinguish cases depending on the connected components of $G$.
	
	\vspace{.1cm}
	{\bf Case 1}. All components of $G$ have at least two vertices. 
	Let $C_1$, $C_2$ a pair of components of $G$ such that $|C_1|\ge3$.
	Take $u_1,v_1,w_1\in C_1$ and $u_2,w_2\in C_2$ such that $u_1v_1,v_1w_1\in E(G)$.
	Then, it is straightforward to check that the $(n-2)$-coloring  $\Pi=\{S_1,\ldots,S_{n-2}\}$ such that $S_1=\{u_1,u_2\}$,  $S_2=\{v_1\}$ and $S_3=\{w_1,w_2\}$, is an NL-coloring.
	Thus, $\chi _{_{NL}}(G)\le n-2$.

	\vspace{.1cm}
	{\bf Case 2}. $G$ has at least two trivial components, i.e., 
	$G$ contains at least two isolated vertices $u$ and $v$.  
	Let $C_1$, $C_2$, $C_3$  be three components of $G$ such that $C_1=\{x\}$, $C_2=\{y\}$ and $z,w \in C_3$.
	Then, it is straightforward to check that the $(n-2)$-coloring  $\Pi=\{S_1,\ldots,S_{n-2}\}$ such that $S_1=\{x,z\}$ and  $S_2=\{y,w\}$, is an NL-coloring.
	Thus, $\chi _{_{NL}}(G)\le n-2$.
	
	\vspace{.1cm}
	{\bf Case 3}. $G$ contains exactly one isolated vertex $u$.  
	Let $H$ be the graph of order $n-1$ without isolated vertices, such that $G\cong H+K_1$ and $V(K_1)=\{u\}$.
	Observe that $\chi _{_{NL}}(H)= \chi _{_{NL}}(G)=n-1$, since  if $\Pi=\{S_1,S_2,\ldots,S_{k}\}$ is a $k$-NL-coloring of $H$,  then the $k$-coloring  $\Pi=\{S'_1,S_2,\ldots,S_{k}\}$ where $S'_1=S_1 \cup\{u\}$, is an NL-coloring of $G$. Thus, according to Theorem  \ref{xnln}, $H$ is a complete multipartite graph.
	\end{proof}

\newpage
\section{Join and disjoint union}\label{ss.join}


This section is devoted to analysing the behavior of the NLC-number with respect to two graph operations: join and disjoint union.

A graph $G=(V,E)$ is  a \emph{join graph} if it is the join $G_1\vee G_2$ of two graphs $G_1=(V_1,E_1)$ and $G_2=(V_2,E_2)$, i.e., if $V=V_1\cup V_2$ and $E=E_1 \cup E_2 \cup E'$, where $E'=\{v_1v_2 :v_1\in V_1, v_2\in V_2 \}$.

Some examples of graphs obtained as the  join of  two graphs are the fan $F_n= K_1 \vee P_{n-1}$, the wheel $W_n=K_1 \vee C_{n-1}$ and the complete bipartite graph $K_{h,k}=\overline{K_h}\vee \overline{K_k}$.

 Clearly, $\chi (G_1\vee G_2) = \chi (G_1)+\chi(G_2)$. 
 In \cite{bean14}, it is shown that, if $G_1$ and $G_2$ are two connected graphs of diameter at most two, then $\chi_{_{L}} (G_1\vee G_2) = \chi_{_{L}} (G_1)+\chi_{_{L}} (G_2)$.
But, in general,  $\chi_{_{L}} (G_1\vee G_2) \geq \chi_{_{L}} (G_1)+\chi_{_{L}} (G_2)$. 
For example, $\chi_{_{L}} (P_{10})= 3$ and $\chi_{_{L}} (P_{10} \vee P_{10}) = 8$ (see \cite{bean14}).

Next, we  study the NLC-number of the join of two graphs.

\begin{remark}
\rm{As a straightforward consequence of the definition, the following properties hold.
If $G_1$,  $G_2$ and $G_3$ are three graphs, then

\begin{enumerate}
\item $G_1 \vee G_2$ is a connected graph of diameter at most 2.
\item $G_1 \vee (G_2 \vee G_3) \cong (G_1 \vee G_2) \vee G_3$.
\end{enumerate}}
\end{remark}

Let $r, n_1,\ldots,n_r,n$ be integers such that $2\le r$, $1 \le n_1 \le \ldots \le n_r$ and $n=n_1+\ldots n_r$.
The complete $r$-partite graph $K_{n_1,\ldots,n_r}$ is the graph
$\overline{K_{n_1}} \vee \ldots \vee \overline{K_{n_r}}$.
In the previous section, we have shown that the NLC-number of a complete multipartite graph  equals the order.
Thus,  $\chi_{_{NL}}(\overline{K_{n_1}} \vee \ldots \vee \overline{K_{n_r}})=\chi_{_{NL}}(K_{n_1,\ldots,n_r})=n_1+ \ldots + n_r=\chi_{_{NL}}(\overline{K_{n_1}} )+ \ldots + (\overline{K_{n_r}})$.
Next theorem extends this result to the join of general graphs.

\begin{theorem}\label{jointhm} For every pair of graphs $G_1$ and $G_2$,
$\chi_{_{NL}} (G_1\vee G_2) = \chi_{_{NL}} (G_1)+\chi_{_{NL}} (G_2).$
\end{theorem}

\begin{proof}
If $\Pi_1=\{S_1, \dots, S_h\}$ is an NL-coloring of $G_1$ and $\Pi_2=\{T_1, \dots, T_k\}$ is an NL-coloring of $G_2$ then, clearly,  $\{S_1, \dots, S_h, T_1, \dots, T_k\}$ is an NL-coloring of $G_1\vee G_2$.

Now, let $\Pi$ be an NL-coloring of $G_1\vee G_2$. 
Observe that, given a vertex $v\in V(G_i)$ and a part $S$ of $\Pi$ such that $v\in S$, then $S\subseteq V(G_i)$ ($i\in\{1,2\}$). 
On the other hand, if $v\in V(G_i)$ and  $S\in \Pi$ such that $S\subseteq V(G_j)$, $i,j\in\{1,2\}$  and $i\neq j$, then $d(v,S)=1$. 
As a consequence, reordering if necessary, we can consider  $\Pi=\{S_1, \dots, S_\ell, S_{\ell+1}, \dots, S_t\}$ so that $\{S_1, \dots, S_\ell\}$ is an NL-coloring of $G_1$ and $\{S_{\ell+1}, \dots, S_t\}$ is an NL-coloring of $G_2$.
\end{proof}


The \emph{disjoint union} of two vertex-disjoint graphs  $G$ and $H$ is the graph denoted by  $G + H$ whose vertex  and edge sets are  $V(G)\cup V(H)$ and  $E(G)\cup E(H)$, respectively.
Next, we present some  properties relating  $\chi_{_{NL}}(G + H)$ to $\ \chi_{_{NL}}(G)$ and $\chi_{_{NL}}(H)$.

\begin{theorem}
	Let $G, H$ be two graphs with $\chi_{_{NL}} (G)=k$ and $\chi_{_{NL}} (H)=h$. The following
	bounds hold for $\chi_{_{NL}} (G + H)$ and are  best possible.
	
	\begin{description}
		\item[(i)] $max\{h,k\}\le \chi_{_{NL}} (G + H)$.
		\item[(ii)] If $G$ has exactly  $k$ isolated  vertices and $H$ has exactly  $h$ isolated vertices,
		then $ \chi_{_{NL}} (G + H)= k+h$;
		\item[(iii)] in any other case,  $ \chi_{_{NL}} (G + H)\le k+h-1$.
		\item[(iv)] If $G$ contains a universal vertex, then  $\chi_{_{NL}} (G + G)\le k+1$.
	\end{description}
	
\end{theorem}
\begin{proof} Any NL-coloring of $G+H$ induces an NL-coloring of  $G$ and an NL-coloring  of $H$,
	thus the first statement is true. For $k\geq 3$ and $h\leq k$,  let
	$G$ and $H$ be the stars $S_{1,k-1}$ and $S_{1,h-1}$, respectively. Since $ \chi_{_{NL}} (S_{1,k-1})=k$,
	$ \chi_{_{NL}} (S_{1,h-1})=h$ and $ \chi_{_{NL}} (S_{1,k-1}+S_{1,h-1})=k$, we have that the 
	given bound is tight. 
	
	To prove {\bf(ii)}, notice that the union of an NL-coloring of $G$ and an NL-coloring
	of $H$ produces an NL-coloring of $G+H$, so  $ \chi_{_{NL}} (G + H)\leq k+h$. 
	On the other hand, since $G+H$ has $k+h$ isolated vertices, we have that $ \chi_{_{NL}} (G + H)\geq k+h$, and the equality follows.
	
	In order to prove {\bf(iii)}, let $\Pi=\{S_1, \dots, S_k\}$ and  $\Pi'=\{S'_1, \dots, S'_h\}$  be
	NL-colorings of $G$ and $H$, respectively.  
	Without loss of generality, we can assume that $G$ has less than $k$ isolated vertices,
	and, therefore, that $S_1$ contains no isolated vertices. 
	In such a case, $\{S_1\cup S'_1, S_2, \dots, S_k,S'_2, \dots, S'_h\}$ is an $(k+h-1)$-NL-coloring
	of $G+H$, establishing the desired bound. 
	To see that this bound is tight, consider the  case $\ell =k+h-1$ of proof of Theorem \ref{rrtt}.

	Finally, to prove  item {\bf(iv)}, observe that from any given NL-coloring of $G$, 
	we can obtain  an NL-coloring of $G+G$
	by  painting the universal vertex of the second copy with a new color $k+1$, and painting any other vertex of the second copy
	with the same color as it has in the first copy. Since $\chi_{_{NL}} (K_k + K_k)\le k+1$, the bound is the best possible.
\end{proof}

\begin{figure}[!ht]
\begin{center}
\includegraphics[width=0.25\textwidth]{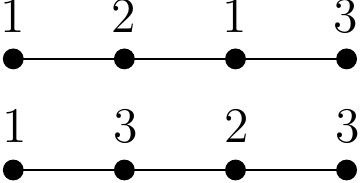}
\caption{An NL-coloring of $2P_4$.}\label{fig_2P4}
\end{center}
\end{figure}

\begin{theorem}\label{rrtt}
Let $h,k$ be integers  such that $3 \le h \le k$.
 Then, for every $\ell\in [k, k+h]$, there exist graphs $G$ and $H$ such that $\chi_{_{NL}} (G)=k$, $\chi_{_{NL}} (H)=h$ and $\chi_{_{NL}}(G + H)=\ell$.
\end{theorem}

\begin{proof}
For the case $\ell=k$,  consider the stars $G\cong S_{1,k-1}$ and $H\cong S_{1,h-1}$.
Then, it is easy to check that $\chi_{_{NL}} (S_{1,k-1} + S_{1,h-1})=\chi_{_{NL}} (S_{1,k-1})=k$.

Case $\ell \in [k+1, k+h-1]$.
Let $G$ be a connected graph of order $k(2^{k-1}-1)$ with $\chi_{_{NL}} (G)=k$ 
(take, for instance, the graph $G_k$ described in Section \ref{bounds})
 and let $H$ be the graph obtained from the complete graph $K_{\ell-k}$ by  hanging  $h-1$ leaves to each of its vertices. 
 Notice that  $\chi_{_{NL}} (H)=h$.
 Take an NL-coloring of $G$ with $k$ colors and notice that it is not possible to color any other  vertex of $G+H$ with these colors.
Let $\{1,\dots, \ell-k\}$ be colors different from the previous ones and  assign these colors  to the vertices of $K_{\ell-k}$.  
Color with $\{1,\dots, \ell-k\} \setminus \{i\}$ the leaves hanging from the vertex of $K_{\ell-k}$ with color $i$. 
In this way, we obtain an NL-coloring of $G + H$ with $k + (\ell-k)= \ell$ colors and it is not possible to do so with less colors.
 Then, $\chi_{_{NL}} (G + H)=\ell$.

 For the case $\ell=k+h$,  take the empty graphs $G\cong \overline{K_k}$ and $H\cong \overline{K_h}$
and observe that  $G + H \cong\overline{K_{k+h}}$ and  $\chi_{_{NL}} (G + H)=k+h$.
\end{proof}

\newpage
\section{Split and Mycielski graphs}\label{ss.split}

This section is devoted to investigating the NLC-number in  two important families of graphs: split graphs and Mycielski graphs.

A \emph{split graph} is a graph such that the vertices can be partitioned into a clique and an independent set.  
When every vertex in the independent set is adjacent to every vertex in the clique it is said to be a \emph{complete split graph}. 
Observe that complete split graphs are examples of complete multipartite graphs that we just studied in the previous section. 
We have taken a step further and we have studied the NLC-number of general connected split graphs.

For any connected split graph $G=(V,E)$ we can assume that  there are two subsets $U$ and $W$ of  $V$ such that

\begin{enumerate}[(i)]
			\item  $V=U\cup W$, $U\cap W=\emptyset$;
			\item $G[U]$ is a complete graph;
			\item $W$ is a maximal independent set, i.e., $W$ is an independent set that for each vertex $u\in U$, there exists a vertex $w\in W$ such that $uw\in E$.
		\end{enumerate}
		
For every $X\subseteq U$, we define  $\mathcal{P}(X)=\{ w\in W : N(w)=X \}$ and  $\rho (G)=\max \{ |X| + |\mathcal{P}(X)|  : X\subseteq U \} $. 
Observe that, by definition, $\rho (G)\, \ge |U|$.

\begin{theorem}
If $G$ is a connected split graph, then
		$$\chi _{_{NL}}(G)=
		\left\{
		\begin{array}{l r}
		\rho (G),&     \hbox{if }  \mathcal{P}(X)= \emptyset \hbox{ for all  }X\subseteq U \hbox{ s.t. } |X|=|U|-1  \\
		\max \{ |U| +1,\rho (G) \},  &\hbox{if }  \mathcal{P}(X)\not= \emptyset \hbox{ for some }X\subseteq U \hbox{ s.t. } |X|=|U|-1\\
		\end{array}
		\right.
		$$
\end{theorem}
\begin{proof}
	%
	First we will prove that $\chi _{_{NL}}(G)$ is at least the given value.
	Assume $\chi _{_{NL}}(G)=k$ and let $\Pi =\{ S_1,\dots ,S_k\}$ be an NL-coloring of $G$.
	We say that $u$ has color $i$ if $u\in S_i$.
	Let $X\subseteq U$. The vertices in $\mathcal{P}(X)$ are pairwise twins and  adjacent to all the vertices in $X$.
	Hence, different vertices of $X\cup \mathcal{P}(X)$ have different  colors.
	Therefore, $$\chi _{_{NL}}(G)\ge \max \{ |X| + |\mathcal{P}(X)|  : X\subseteq U \} \,=\rho (G).$$

In addition, if there is a set $X\subseteq U$ such that $|X|=|U|-1$ and $\mathcal{P}(X)\not= \emptyset$, we claim that
$\x (G)\ge |U|+1$.
Indeed,  if $\rho (G)\ge |U|+1$, then the assertion is obvious. Otherwise, $\rho (G)=|U|$, which implies  $\mathcal{P} (U)=\emptyset$ and
	$|\mathcal{P} (X)|=1$. Let $w\in W $ be the  only vertex   in $\mathcal{P}(X)$ and let $u\in U $ be the  only vertex in $U \setminus X$.
Since $N(u) \cap N(w)$ induces a complete graph of size $ |U|-1$, we have that   $\chi _{_{NL}}(G) \ge  |U| +1$.

	
	%
	Now we will prove that $\chi _{_{NL}}(G)$ is at most the given value.
	
	%
	First, suppose that $\mathcal{P}(X)= \emptyset$ for all $X\subseteq U$ such that  $|X|=|U|-1$.
	We  construct an NL-coloring
	$\Pi =\{ S_1,\dots ,S_k\}$, where
	$$k=\rho (G)= \max \{ |X| + |\mathcal{P}(X)|  : X\subseteq U \}.$$

Let $U=\{ x_1,\dots ,x_r\}$ and color each  $x_i$ with $i$ for $1\leq i \leq r$. Notice that this is possible since $k=\rho (G)\ge |U|$.
Then, for every set $X\subseteq U$,   color the elements of $\mathcal{P}(X)$  using  $| \mathcal{P} (X) |$ distinct colors chosen from the ones not used to color the vertices of $X$.
	Notice that this is possible since $k- |X|\ge |\mathcal{P}(X)|$.


	We claim that $\Pi$ is an NL-coloring.
	Indeed, two vertices with the same color are non adjacent by construction.
	Now, suppose that $u$ and $v$ have the same color $i$, for some $i\in \{ 1,\dots ,k\}$.
	We  prove that $u$ and $v$ are neighbor-located by $\Pi$.
	We consider two cases:  either $u,v\in W$; or $u\in U$ and $v\in W$.

		In the first case,  to prove that $u$ and $v$ are neighbor-located by $\Pi$, it is enough to see that $N(u)\neq N(v)$. 
		In fact, if $N(u)=N(v)=X$, then $u$ and $v$ belong to $\mathcal{P}(X)$, and thus $u$ and $v$ have different colors, contradicting the assumption.

	In the second case,
	assume that $v\in \mathcal{P}(X)$.
	Then, $u\notin X=N(v)$, and so, $|X|\le |U|-2$.
	The neighbors of $v\in \mathcal{P}(X)$ are colored with exactly $|X|$ different colors, with $|X| \le |U|-2$.
	However, the neighbors of $u$ are colored with at least $|U|-1$ different colors.
	Hence, $u$ and $v$ are neighbor-located by $\Pi$.
	%
	
	%
	Now, suppose that $\mathcal{P}(X)\not= \emptyset$ for some subset $X\subseteq U$ such that  $|X|=|U|-1$.
	We  construct an NL-coloring
	$\Pi =\{ S_1,\dots ,S_k\}$, where $k= \max \{\rho (G), |U|+1\}$.

As before, let $U=\{ x_1,\dots ,x_r\}$ and  color each  $x_i$ with $i$ for $1\leq i \leq r$. Again, this is possible since  $k\ge \rho (G)\ge |U|$.

	Let $U'=\{ u\in U : \mathcal{P}(U\setminus \{ u \})\not= \emptyset \}$ and take a maximal twin-free subset $W'$ of $N(U')\cap W$. Notice that, such a set exists and, by construction, $U'\subseteq N(W')$. Color all the vertices of $W'$ with color $r+1$.
		Finally, for every subset $X \subseteq U$,
	 color the vertices in $\mathcal{P}(X) \setminus W'$ with different colors chosen from the ones not used to color the vertices of $X$, and without using color $r+1$.
	Notice that it is possible since $k- |X|\ge |\mathcal{P}(X)|$.

	We claim that $\Pi$ is an NL-coloring.
	Indeed, vertices with the same color are non adjacent by construction.
	Now suppose that $u$ and $v$ have the same color $i$, for some $i\in \{ 1,\dots ,k\}$.
	Without loss of generality, we may distinguish two cases:  $u,v\in W$; or $u\in U$ and $v\in W$.
	
	In the first case, $u\in \mathcal{P} (X_u)$ and $u\in \mathcal{P} (X_v)$, where $N(u)=X_u\not=N(v)=X_v$, for some $X_u,X_v\subseteq U$.
	Therefore, the set of colors of the neighbors of $u$ and $v$ are different.
	Hence, $u$ and $v$ are  neighbor-located by $\Pi$.
	
	In the second case,
	assume that $v\in \mathcal{P}(X)$.
	Then, $u\notin X=N(v)$, and so $|X|\le |U|-1$.
	If $|X|\le |U|-2$, the neighbors of $v\in \mathcal{P}(X)$ are colored with exactly $|X|\le |U|-2$ different colors.
	However, the neighbors of $u$ are colored with at least $|U|-1$ different colors.
	Hence, $u$ and $v$ are neighbor-located by $\Pi$.
	
	If $|X|\le |U|-1$, the neighbors of $v\in \mathcal{P}(X)$ are colored with exactly $|X|$ different colors from $\{1,\dots, r\}$
	However, at least one neighbor of $u$ has color $r+1$.
	Hence, $u$ and $v$ are neighbor-located by $\Pi$.
\end{proof}

\begin{remark}
\rm{	The value of the NLC-number  obtained for general split graphs  fits with some known results, such as $\chi _{_{NL}}(S_{1,n-1})=n$, $\chi _{_{NL}}(K_n)=n$, and $\chi _{_{NL}}(G)=n$, whenever $G$ is a complete split graph.}
\end{remark}

Another interesting class of graphs are Mycielski graphs.
Given a graph $G$ of order $n$, the Mycielski graph $\mu(G)$ of $G$ is a graph of order $2n+1$ that contains $G$ as an induced
subgraph. Concretely, if $V(G)=\{v_1, \dots, v_n\}$, then $V(\mu(G))=\{v_1, \dots, v_n\} \cup \{u_1, \dots, u_n\} \cup \{w\}$ and $E(\mu(G))=E(G) \cup \{wu_i \, : \, 1\leq i\leq n\} \cup \{v_i u_j \, : \, 1\leq i, j\leq n, \, v_iv_j \ E(G)\}$ (see Figure \ref{myc}). 
Thus,  $|E(\mu(G))|=3\, |E(G)|+n$.
From now on, we will use this terminology when referring to the set of vertices of a graph $G$ and its Mycielski graph $\mu (G)$.

\begin{figure}[!ht]
	\begin{center}
		\includegraphics[width=0.70\textwidth]{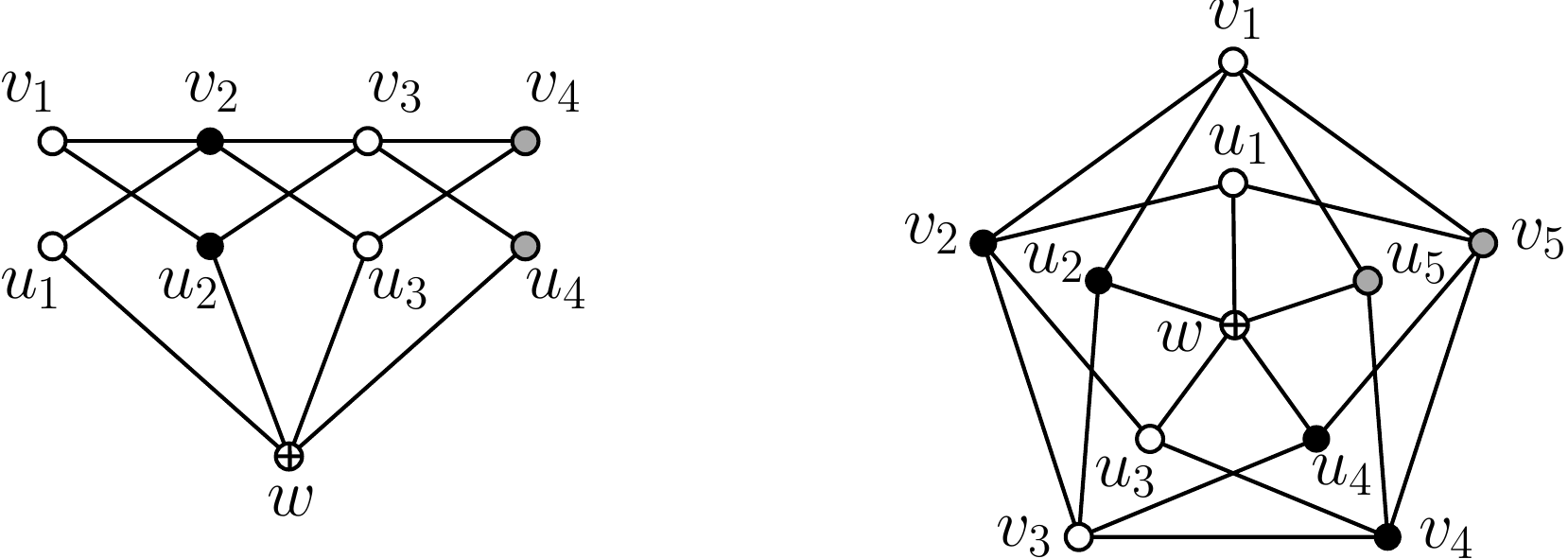}
		\caption{A pair of Mycieslki graphs.
		Left: $\mu(P_5)$. Right: $\mu(C_5)$.}.
		\label{myc}
	\end{center}
\end{figure}

Mycielski \cite{myci} designed these graphs to prove that it is possible to increase the chromatic number of a graph without increasing the clique number. 
More precisely,  $\omega(\mu(G))=\max(2,\omega(G))$ and  $\chi(\mu(G))=\chi(G)+1$.

Next, we give a  similar partial result for the NLC-number.

\begin{prop} \label{p:myci} 
For every graph $G$, $\chi_{_{NL}}(\mu(G))\leq \chi_{_{NL}}(G)+1$.
\end{prop}
\begin{proof}
It is sufficient to prove that every $k$-NL-coloring  $\Pi=\{S_1, \dots,S_k\}$ of $G$ can be extended  to a $(k+1)$-NL-coloring of $\mu(G)$. 
Indeed, let  $\Pi'=\{S'_1, \dots,S'_k, S'_{k+1}\}$, where $S'_h=S_h \cup \{u_i \, : \, v_i\in S_h\}$,  if $1\leq h\leq k$, and $S'_{k+1}= \{w\}$. 
By definition of $\mu (G)$, the sets $S_i'$ are independent in $\mu (G)$, for every $i\in \{1,\dots, k\}$.
Besides, the set of colors of the neighborhood of $v_i$ in $\mu (G)$ is the same as for $v_i$ in $G$; and the set of colors of the neighbors of $u_i$ in $\mu (G)$ is the same as for $v_i$ in $G$ together with color $k+1$. 
Hence, vertices of $\mu (G)$ with the same color $i\in \{1,\dots ,k\}$ in $\mu (G)$, have neighborhoods with different sets of colors. 
From here, the inequality follows.
\end{proof}    

We show next that this bound is tight.

\begin{prop} \label{p:3}
If $G$ is a complete multipartite graph, then $\chi_{_{NL}}(\mu(G))= \chi_{_{NL}}(G)+1$.
\end{prop}
\begin{proof}
By Theorem \ref{xnln}, we know that $G$ is a complete multipartite graph of order $n$ if and only $\chi_{_{NL}}(G)= n$.
In order to derive a contradiction, assume that there exists a  $NL$-coloring $\Pi=\{S_1, \dots,S_n\}$ of $\mu(G)$ using  $n$ colors.
Without loss of generality,  assume that $w \in S_1$.
Thus,  vertices $u_1,\dots ,u_n$ are colored using at most $n-1$ colors, so that there are 2 vertices with the same color, say $u_1,u_2\in S_2$. 
Hence, vertices $v_1$ and $v_2$ must be adjacent in $G$, otherwise $u_1$ and $u_2$ would be false twins in $\mu (G)$, which is a contradiction, because both vertices have the same color in $\mu (G)$.
Hence, $\{ v_1,\dots ,v_n\}\subseteq N(u_1)\cup N(u_2)$, which implies that no vertex in $\{ v_1,\dots ,v_n\}$ has color $2$. 
Thus, the $n$ vertices $v_1,\dots ,v_n$ are colored using at most $n-1$ colors in $\mu(G)$, which in turn implies that there exist two vertices, $v_{i_1}$ and $v_{i_2}$ with a same color. Thus, $v_{i_1}$ and $v_{i_2}$ must be non-adjacent in $G$, implying that $v_{i_1}$ and $v_{i_2}$ are false twins with the same color, which is again a contradiction.
\end{proof}

\newpage
\section{Concluding remarks and open problems}\label{op}

In this paper, we have introduced the neighbor-locating chromatic number of a graph, parameter that 
measures the minimum  number of colors needed to paint a graph in such a way that any two vertices
with the same color can be differentiated by the set of colors used by its neighbors. 
We believe that  this new parameter will play a significant roll in the study of the structure of a graph, \textit{ per se} and by comparison with other previous known parameters such as the metric-locating chromatic number  and  the partition metric-location-domination number. 

In view of the results obtained in the present paper and in the simultaneous work  \cite{aghmp19}(where we 
focus our attention on determining the neighbor-locating chromatic number of paths, cycles, fans, wheels and unicyclic graphs),
we propose the following conjectures as future development directions in the study of  neighbor-locating partitions of a graph.

The following conjecture arises from the results obtained in Sections \ref{lp} and \ref{bounds}.
	\begin{conj}
			For each pair $h$, $k$ of integers with $3 \le  h \le k$,
			there exists a connected graph $G$ such that $\chi _{_{L}}(G)=h$ and $\chi _{_{NL}}(G)=k$.
	\end{conj}

In Section \ref{ss.join}, we have  addressed the behavior
 of the neighbor-locating chromatic number with respect to the join and the disjoint union of graphs.
 The following analysis  is related to this behavior with respect to  the Cartesian product $G\square H$
   and the lexicographic product $G[H]$  of two arbitrary graphs $G$ and $H$ \cite{hik11}.

Given an NL-coloring of $G$ with  colors $A=\{ a_1,\dots ,a_r \}$
 and an NL-coloring of $H$ with disjoint set of colors $B=\{ b_1,\dots ,b_s \}$,  
 consider the  coloring of $G\square H$ using the set of colors  
 $A\times B$ as follows: assign color $(a_i,b_j)$ to a vertex $(u,v)$ if $u$ has 
 color $a_i$ in $G$ and $v$ has color $b_j$ in $H$.
It is an easy exercise to prove that this is an NL-coloring of $G\square H$.
Thus,
$$\chi_{_{NL}}(G\square H)\le\chi_{_{NL}} (G) \, \chi_{_{NL}} (H).$$
In some way, this bound is  best possible since the  equality holds,
 for instance, when $G\cong H\cong P_2$.

With regard to the lexicographic product, it is easy to check that
$$\omega (G) \, \omega (H)\le \chi_{_{NL}} (G[H])\le\chi_{_{NL}} (G) \, \chi_{_{NL}} (H),$$
where $\omega (G)$ denotes the clique number of $G$. 
However, we believe that, in general, these bounds are far from being tight.
It is therefore an open problem  to find tighter bounds.
Another interesting problem is to determine the neighbor-locating chromatic number
 of both the Cartesian product and the lexicographic product of two graphs when one of them is 
  a path or a complete graph or a cycle.
  
We propose the following conjecture involving  both operations.

\begin{conj}
If $G$ and $H$ are connected graphs, then
$\chi _{_{NL}} (G[H]) \le \chi _{_{NL}} (G\square H)$.
\end{conj}

In Section \ref{ss.split}, we dealt with the problem of determining
the neighbor-locating chromatic number of Mycielski graphs $\mu(G)$. 
We have shown that in general $\chi_{_{NL}}(\mu(G))\leq \chi_{_{NL}}(G)+1$, but conjecture that the
equality holds for any graph $G$.

\begin{conj} \label{conj:mycie}
For any graph $G$, $\chi_{_{NL}}(\mu(G))= \chi_{_{NL}}(G)+1$.
\end{conj}

Proposition \ref{p:3}  
supports this conjecture for the class of complete multipartite graphs.


\end{document}